\documentclass[12pt,reqno]{amsart}
\usepackage{amssymb}
\usepackage{epsfig}
\usepackage{graphicx}

\begin{document}

\newtheorem{mthm}{Theorem}
\newtheorem{mcor}{Corollary}
\newtheorem{mpro}{Proposition}
\newtheorem{mfig}{figure}
\newtheorem{mlem}{Lemma}
\newtheorem{mdef}{Definition}
\newtheorem{mrem}{Remark}
\newtheorem{mpic}{Picture}
\newtheorem{rem}{Remark}[section]
\newcommand{\ra}{{\mbox{$\rightarrow$}}}
\newtheorem{thm}{Theorem}[section]
\newtheorem{pro}{Proposition}[section]
\newtheorem{lem}{Lemma}[section]
\newtheorem{defi}{Definition}[section]
\newtheorem{cor}{Corollary}[section]

\title[]{Further study on periodic solutions   of  elliptic equations with a fractional Laplacian}

\author{ Zhuoran Du$\dag$}
 \footnotetext{$\dag$College of Mathematics and
Econometrics, Hunan University, Changsha 410082,
                 PRC.
 E-mail: {duzr@hnu.edu.cn}}

\author{Changfeng Gui$^\ast$}
 \footnotetext{$^\ast$Department of Mathematics,  University of Texas at San Antonio, TX78249, USA.
E-mail: {changfeng.gui@utsa.edu}}

\date{}
\maketitle

\begin{abstract}
We obtain some existence theorems for periodic solutions to several linear equations involving fractional Laplacian.
We also prove that the lower bound of all periods for semilinear elliptic equations  involving fractional Laplacian is not larger than some exact positive constant. Hamiltonian identity, Modica-type  inequalities and an estimate of the energy for periodic solutions are also established.
\end{abstract}



\section{Introduction}

 We first consider the following  linear problem involving  fractional  Laplacian
\begin{equation}\label{1}
  Lu:=(-\partial_{xx})^s  u(x)  + k(x)u(x)=g(x)\indent
        \mbox{ in } \mathbb{R},
\end{equation}
where $k, g$ are periodic functions with the same period $T$ and $k$ is bounded in $\mathbb{R}$.
Here $(-\partial_{xx})^s$, $ s\in(0,1)$, denotes the usual fractional
Laplace operator, a Fourier multiplier of symbol $|\xi|^{2s}$.

The fractional Laplace operator $(-\Delta)^{s}$ can  be defined as a
Dirichlet-to-Neumann map for a so-called $s$-harmonic extension
problem (see \cite{caff}).
Given a function $\phi(x)$, the solution $\Phi(x,y)$ of the following problem
$$
\begin{cases}
\mbox{div}(y^a\nabla    \Phi)=0 \indent &\mbox{in}~~
\mathbb{R}_+^{n+1}=\{(x,y): x\in\mathbb{R}^{n}, y>0\},\\
\Phi(x,0)=\phi(x)\indent
&\mbox{on}~~\mathbb{R}^{n}\\
\end{cases}
$$
is called the $s$-harmonic extension of $\phi$, and we denote it as $\Phi:=\mbox{Ext}(\phi)$. The parameter $a$ is related to the power $s$ of the fractional Laplacian $(-\Delta)^{s}$ by the formula
$a=1-2s\in(-1,1)$.
One has
\begin{equation}\label{2}
\Phi(x,y)=\int_{\mathbb{R}^{n}}p_s(x-z,y)\phi(z)dz=\int_{\mathbb{R}^{n}}p_s(z,y)\phi(x-z)dz,
\end{equation}
where $p_s(x,y)$ is the $s$-Poisson kernel
$$
p_s(x,y)=C_{n,s}\frac{y^{2s}}{(|x|^2+|y|^2)^{\frac{n+2s}{2}}},
$$
and $C_{n,s}$ is the constant which makes $\int_{\mathbb{R}^{n}}p_s(x,y)dx=1$.
Caffarelli and Silvestre in \cite{caff} proved that
$$
(-\Delta)^{s}\phi(x)=d_s\frac{\partial  \Phi}{\partial \nu^a}   \indent~~ \mbox{in}~~~~~\mathbb{R}^{n}=\partial \mathbb{R}^{n+1}_+,
$$
where
$$
\frac{\partial  \Phi}{\partial \nu^a}:=-\lim_{y\downarrow
0}y^{a}\frac{\partial  \Phi}{\partial y},~~d_s=2^{2s-1}\frac{\Gamma(s)}{\Gamma(1-s)}.
$$

From (\ref{2}) and the  formula of $s$-Poisson kernel, we can easily deduce that the $s$-harmonic extension $\mbox{Ext}(u)(x,y)$ of an odd (resp. even) periodic function $u(x)(x\in\mathbb{R})$  is also odd (resp. even) and periodic with respect to the variable $x$ with the same period as of $u$.

We will first establish several existence theorems of periodic solutions to linear problems relevant to (\ref{1}).

We also consider  the following semilinear equation
\begin{equation}\label{3}
    (-\partial_{xx})^s  u(x)+F'(u(x))=0,\indent
       u(x+T)=u(x)~~~ \mbox{ in } \mathbb{R}.
 \end{equation}
Here the  function $F$ is a smooth double-well potential with wells at $+1$ and $-1$, namely, it satisfies
\begin{eqnarray}
\label{4}
 \left\{\begin{array}{l} F(1)=F(-1)=0<F(u),\indent
 \forall -1<u<1,\\
F'(1)=F'(-1)=0.\\
\end{array}
\right.
\end{eqnarray}
We also assume that \begin{equation}\label{5}
F \mbox{ is nondecreasing in }  (-1,0) \mbox{ and nonincreasing in } (0,1).
 \end{equation}

 The authors and  Zhang in \cite{gzd} obtained the existence of periodic solutions with any period $T>T_0$ to (\ref{3}), by using variational methods.
An estimate of energy for periodic solutions also had been established.
In \cite{ddgw} the authors study Delaunay-type singular solutions for the fractional Yamabe problem with an isolated singularity at the origin.
Precisely, after  the Emden-Fowler change of variale, they reformulate their problem into a variational one for  some periodic function with period $L$.
Then they prove that there is the smallest period $L_0$, namely the periodic problem admits nonconstant minimizer for any $L>L_0$.
Existence of periodic solutions to so-called pesudo-relativistic Schr\"{o}dinger equations are also established in \cite{Ambrosio1} and \cite{Ambrosio2}.
In \cite{roncal},the authors establish interior and boundary Harnack's inequalities for nonnegative solutions to $(-\Delta)^s  u=0$ with periodic boundary conditions,
and they also obtain regularity properties of the fractional Laplacian with periodic boundary conditions and the pointwise integro-differential formula for the operator.

We will obtain an exact upper bound value of $T_0$ by using Hopf bifurcation theory in section 3.
  Hamiltonian identity and Modica-type  inequalities for periodic solutions of (\ref{3}) will also be established in section 4.
Finally in section 5 we will improve the estimate of the energy of periodic solutions to (\ref{3})  in \cite{gzd}.

\section{Existence theorems of linear  problems}

We introduce the following spaces
\begin{eqnarray}
\nonumber &
\mathcal{H_T}:=&\{U(x,y): U(x+T,y)=U(x,y),~\forall y\geq0,
\\ \nonumber &&
\|U\|^2_T:=\int_{-\frac{T}{2}}^{\frac{T}{2}}\int_{\mathbb{R}^+}y^a|\nabla U(x,y)|^2dxdy
+\int_{-\frac{T}{2}}^{\frac{T}{2}}U^2(x,0))dx<\infty\},
\end{eqnarray}
\begin{eqnarray}
\nonumber &
H_T^s:=&\{u(x): u(x+T)=u(x),
\\ \nonumber &&
\|u\|^2_{H^s_T}:=\int_{-\frac{T}{2}}^{\frac{T}{2}}\int_{-\frac{T}{2}}^{\frac{T}{2}}
\frac{|u(x)-u(\bar{x})|^2}{|x-\bar{x}|^{1+2s}}dxd\bar{x}+
\int_{-\frac{T}{2}}^{\frac{T}{2}}u^2(x)dx<\infty\},
\end{eqnarray}
\begin{eqnarray}
\nonumber &
\tilde{H}_T^s:=&\{u(x)\in H_T^s, ~
\|u\|^2_{\tilde{H}^s_T}:=\int_{-\frac{T}{2}}^{\frac{T}{2}}u^2(x)dx
\\ \nonumber &&
+\int_{-\frac{T}{2}}^{\frac{T}{2}}\int_{-\frac{T}{2}}^{\frac{T}{2}}
\sum_{j\in Z, j=-\infty}^{+\infty}\frac{|u(x)-u(\bar{x})|^2}{|x-\bar{x}+jT|^{1+2s}}dxd\bar{x}
<\infty\},
\end{eqnarray}
$$
L_T^2:=\{u(x): u(x+T)=u(x), ~\|u\|^2_{L^2_T}:=
\int_{-\frac{T}{2}}^{\frac{T}{2}}u^2(x)dx<\infty\}.
$$
Clearly
\begin{equation}\label{6}
\tilde{H}_T^s\hookrightarrow H_T^s.
\end{equation}
Note that for periodic functions $u$, we have
$$
\int_{-\frac{T}{2}}^{\frac{T}{2}}\int_{-\frac{T}{2}}^{\frac{T}{2}}
\sum_{j=-\infty}^{+\infty}\frac{|u(x)-u(\bar{x})|^2}{|x-\bar{x}+jT|^{1+2s}}dxd\bar{x}=\int_{-\frac{T}{2}}^{\frac{T}{2}}\int_{-\infty}^{\infty}
\frac{|u(x)-u(\bar{x})|^2}{|x-\bar{x}|^{1+2s}}dxd\bar{x}.
$$

We claim
\begin{equation}\label{7}
\mbox{ Tr}(\mathcal{H_T})=\tilde{H}_T^s.
\end{equation}
Indeed, for any given $u\in \tilde{H}_T^s$, we may first prove $\mbox{Ext}(u)\in \mathcal{H_T}$, which means the surjectivity of the trace operator(if it is well defined) from $\mathcal{H_T}$ to $\tilde{H}_T^s$.
Since $\mbox{Ext}(u)$ is periodic with respect to the variable $x$ for any $y\geq0$, we have
$$
\int_{-\frac{T}{2}}^{\frac{T}{2}}\int_{\mathbb{R}^+}y^a|\nabla \mbox{Ext}(u)(x,y)|^2dxdy=
\frac{1}{d_s}\int_{-\frac{T}{2}}^{\frac{T}{2}}u(x) (-\partial_{xx})^su(x)dx.
$$
In view of
$$
(-\partial_{xx})^su(x)=C(s)\int_{-\infty}^{\infty}
\frac{u(x)-u(\bar{x})}{|x-\bar{x}|^{1+2s}}d\bar{x},
$$
we have
\begin{eqnarray}
\nonumber
&&\int_{-\frac{T}{2}}^{\frac{T}{2}}u(x) (-\partial_{xx})^su(x)dx=C(s)\int_{-\frac{T}{2}}^{\frac{T}{2}}\int_{\mathbb{R}}\frac{u(x)[u(x)-u(\bar{x})]}{|x-\bar{x}|^{1+2s}}d\bar{x} dx
\\
\nonumber
&&=\frac{C(s)}{2}\left\{\int_{-\frac{T}{2}}^{\frac{T}{2}}\int_{\mathbb{R}}\frac{u(x)[u(x)-u(\bar{x})]}{|x-\bar{x}|^{1+2s}}d\bar{x} dx
+\int_{-\frac{T}{2}}^{\frac{T}{2}}\int_{\mathbb{R}}\frac{u(\bar{x})[u(\bar{x})-u(x)]}{|x-\bar{x}|^{1+2s}}d\bar{x} dx\right\}
\\
\nonumber
&&=\frac{C(s)}{2}\int_{-\frac{T}{2}}^{\frac{T}{2}}\int_{\mathbb{R}}\frac{|u(x)-u(\bar{x})|^2}{|x-\bar{x}|^{1+2s}}d\bar{x} dx,
\end{eqnarray}
where we used the fact that $u$ is periodic.
Hence
\begin{eqnarray}
\nonumber &&
\int_{-\frac{T}{2}}^{\frac{T}{2}}\int_{-\frac{T}{2}}^{\frac{T}{2}}
\sum_{j=-\infty}^{+\infty}\frac{|u(x)-u(\bar{x})|^2}{|x-\bar{x}+jT|^{1+2s}}dxd\bar{x}
\\ \nonumber &&
=\frac{2d_s}{C(s)}\int_{-\frac{T}{2}}^{\frac{T}{2}}\int_{\mathbb{R}^+}y^a|\nabla \mbox{Ext}(u)(x,y)|^2dxdy,
\end{eqnarray}
which yields that $c\|\mbox{Ext}(u)\|_T\leq\|u\|_{\tilde{H}^s_T}\leq C\|\mbox{Ext}(u)\|_T$.  Therefore we have  $\mbox{Ext}(u)\in \mathcal{H_T}$.

Next we prove that for any $U(x,y)\in \mathcal{H_T}$ satisfying $U(x,0)=u(x)$,  the following inequality holds
\begin{equation}\label{ineq}
\int_{-\frac{T}{2}}^{\frac{T}{2}}\int_{-\frac{T}{2}}^{\frac{T}{2}}
\sum_{j=-\infty}^{+\infty}\frac{|u(x)-u(\bar{x})|^2}{|x-\bar{x}+jT|^{1+2s}}dxd\bar{x}
\leq\frac{2d_s}{C(s)}\int_{-\frac{T}{2}}^{\frac{T}{2}}\int_{\mathbb{R}^+}y^a|\nabla U(x,y)|^2dxdy.
\end{equation}
Since we have
$$
\int_{-\frac{T}{2}}^{\frac{T}{2}}u(x) (-\partial_{xx})^su(x)dx=
d_s\int_{-\frac{T}{2}}^{\frac{T}{2}}\int_{\mathbb{R}^+}y^a\nabla \mbox{Ext}(u)(x,y)\cdot \nabla U(x,y) dxdy,
$$
then
\begin{eqnarray}
\nonumber &&
\frac{C(s)}{2d_s}\int_{-\frac{T}{2}}^{\frac{T}{2}}\int_{\mathbb{R}}\frac{|u(x)-u(\bar{x})|^2}{|x-\bar{x}|^{1+2s}}d\bar{x} dx
\\ \nonumber &&
\leq \left(\int_{-\frac{T}{2}}^{\frac{T}{2}}\int_{\mathbb{R}^+}y^a|\nabla \mbox{Ext}(u)(x,y)|^2dxdy\right)^{\frac{1}{2}}
\left(\int_{-\frac{T}{2}}^{\frac{T}{2}}\int_{\mathbb{R}^+}y^a|\nabla U(x,y)|^2dxdy\right)^{\frac{1}{2}}
\\ \nonumber &&
=\sqrt{\frac{C(s)}{2d_s}} \left(\int_{-\frac{T}{2}}^{\frac{T}{2}}\int_{\mathbb{R}}\frac{|u(x)-u(\bar{x})|^2}{|x-\bar{x}|^{1+2s}}d\bar{x} dx\right)^{\frac{1}{2}}
\left(\int_{-\frac{T}{2}}^{\frac{T}{2}}\int_{\mathbb{R}^+}y^a|\nabla U(x,y)|^2dxdy\right)^{\frac{1}{2}},
\end{eqnarray}
which yields the above desired inequality \eqref{ineq}. Therefore the trace operator is well defined.

Hence we have proved claim (\ref{7}).

From now on, we denote the $s$-harmonic extension $\mbox{Ext}(u)$ of $u$  as $U$ for simplicity of notation.

Set
$$
B(U,V):=\int_{\mathbb{R}^+}\int_{-\frac{T}{2}}^{\frac{T}{2}}y^a\nabla U(x,y)\cdot \nabla V(x,y)dxdy
+\int_{-\frac{T}{2}}^{\frac{T}{2}}k(x)U(x,0)V(x,0)dx.
$$
We denote
$$
B_\mu(U,V):=B(U,V)+\mu \int_{-\frac{T}{2}}^{\frac{T}{2}}U(x,0)V(x,0)dx
.$$
Clearly we see that
$
B_\mu(U,V)=B_\mu(V,U),
$
and there exists $\gamma\geq0$ such that
$$
B_\gamma(U,U)=B(U,U)+\gamma\|U(x,0)\|^2_{L^2_T}\geq C\|U\|^2_{T}.
$$
Hence
$$
(\cdot,\cdot):=B_\gamma(\cdot,\cdot)
$$
is an inner product in $\mathcal{H_T}$.

\begin{thm}\label{thm1}(First existence Theorem)
There exists $\gamma\geq0$ such that for each  $\mu\geq \gamma$, and each $g\in L^2_T$,   the problem
 \begin{equation}\label{8}
  L_\mu u:=Lu+\mu u=g\indent
        \mbox{ in } \mathbb{R},
\end{equation}
admits a unique weak periodic solution $u:=L^{-1}_{\mu} g\in H^s_T$. Moreover,
the following estimate holds
$$
\|u\|_{H^s_T}\leq C\|g\|_{L^2_T}.
$$
\end{thm}
\begin{proof}
We apply Riesz Representation Theorem to find a unique solution $U\in \mathcal{H_T}$ of
$$
B_\mu(U,V)=\langle g, V(x,0)\rangle, ~~\forall V\in \mathcal{H_T}.
$$
Moreover
$$
\|U\|_{T}\leq C\|g\|_{L^2_T}.
$$
Consequently $U(x,y)$ is the unique  weak periodic solution (periodic with respect to the variable $x$) of
$$
\begin{cases}
\mbox{div}(y^a\nabla    U)=0&\mbox{in}~~~~~
\mathbb{R}_+^{2},\\
\frac{\partial U}{\partial \nu^a}+k(x)U+\mu U=g(x)&\mbox{on}~~\partial \mathbb{R}_+^{2},\\
\end{cases}
$$
which means that $u(x)=U(x,0)$ is the unique periodic solution of (\ref{8}).
By (\ref{6})-(\ref{7}), we have
$$
\|u\|_{H^s_T}\leq C\|u\|_{\tilde{H}^s_T} \leq C\|U\|_{T}.
$$
So we derive
$$
\|u\|_{H^s_T}\leq C \|g\|_{L^2_T}.
$$
\end{proof}

We shall show the following Fredholm alternative.

\begin{thm}\label{thm2}(Second existence Theorem)
One of the following statements holds:

either

(i)for each  $g\in L^2_T$, there exists  a unique weak periodic  solution $u$ of
\begin{equation}\label{9}
Lu=g.
\end{equation}

or else

(ii)there exists a weak solution $u\not\equiv0$ of
\begin{equation}\label{10}
Lu=0.
\end{equation}

Furthermore, should (ii) hold, the dimension of the subspace $\mathcal{N}\subset L^2_T$ of weak solutions of  (\ref{10}) is finite.

Finally,  (\ref{9}) has a weak solution if and only if $\langle g,v\rangle=0, \forall v\in \mathcal{N}.$
\end{thm}
\begin{proof}
Observe that (\ref{9}) is equivalent to the problem
\begin{equation}\label{11}
L_\gamma u=\gamma u+g.
\end{equation}
This equation can be written as
$$
u=L^{-1}_{\gamma}(\gamma u+g)=\gamma L^{-1}_{\gamma} u+L^{-1}_{\gamma}g.
$$
Denote
$$
A(u):=\gamma L^{-1}_{\gamma} u,~~~\hat{g}:= L^{-1}_{\gamma} g.
$$
From Theorem \ref{thm1} and the compact embedding of $H^s_T$ into $L^2_T$, we know that $A$ is a compact operator from $L^2_T$ into itself.
Moreover, the operator $A$ is self-adjoint.
Applying the Fredholm alternative, we obtain:

either

(i) for each $\hat{g}\in L^2_T$ the equation
 \begin{equation}\label{12}
u-Au=\hat{g}
 \end{equation}
 has a unique solution $u$

or else

(ii) the equation
\begin{equation}\label{13}
u-Au=0
 \end{equation}
  has nonzero solutions.

Should assertion (i) hold, then $u$ is the unique weak solution  of  (\ref{9}).
On the other hand, should assertion (ii) be valid, then necessarily $\gamma\neq0$ and it is well known that the dimension of the space $\mathcal{N}$
of the solutions is finite.
We check that (\ref{13}) holds if and only if $u$ is a weak solution of (\ref{10}).

Finally, we recall (\ref{12}) has a solution if and only if
$$
\langle \hat{g},v\rangle=0,~~\forall v\in \mathcal{N}.
$$
Note
$$
\langle \hat{g},v\rangle=\frac{1}{\gamma}\langle Ag, v\rangle=\frac{1}{\gamma}\langle g, Av\rangle=\frac{1}{\gamma}\langle g, v\rangle.
$$
Therefore (\ref{9}) has a weak solution if and only if $\langle g,v\rangle=0$  for any $ v\in \mathcal{N}.$
\end{proof}

We now state the result regarding the associated eigenvalue problem.

\begin{thm}\label{thm3}(Third existence Theorem)

(i) There exists an at most countable set of eigenvalues  $\Lambda\subset \mathbb{R}$ such that
\begin{equation}\label{14}
Lu=\lambda u+g
 \end{equation}
admits a unique weak periodic solution $u$  for each $g\in L^2_T$ if and only if
 $\lambda\not\in\Lambda$.

(ii) If $\Lambda$ is infinite, then $\Lambda=\{\lambda_j\}_{j=1}^\infty$, the values of a nondecreasing sequence with
$$
\lambda_j\rightarrow+\infty.
$$
\end{thm}
\begin{proof}
From Theorem \ref{thm1} we know that (\ref{14}) admits weak solutions for each $g\in L^2_T$ when $\lambda\leq-\gamma$.
Hence we only need to consider $\lambda>-\gamma$.
Assume also with no loss of generality that $\gamma>0$.

According to the Fredholm alternative, (\ref{14}) has a unique weak solution for each $g\in L^2_T$
if and only if $u\equiv0$ is the only weak solution of the homogeneous problem $Lu=\lambda u$.
This is in turn true if and only if $u\equiv0$
is the only weak solution of
$$
L_\gamma u=(\gamma+\lambda)u.
$$
This holds exactly when
\begin{equation}\label{15}
u=L^{-1}_{\gamma}(\gamma+\lambda) u=\frac{\gamma+\lambda}{\gamma}Au,
 \end{equation}
where, as in the proof of Theorem 2.2, we have set $Au=\gamma L^{-1}_{\gamma}u$.
Here we recall  that $A: L^2_T\rightarrow L^2_T$ is a bounded, linear compact operator.

Now if $u\equiv0$ is the only solution of (\ref{15}), we know that
\begin{equation}\label{16}
\frac{\gamma}{\gamma+\lambda} ~\mbox{is not an eigenvalue of}~ A.
 \end{equation}

Consequently we see (\ref{14}) has a unique weak solution for each $g\in L^2_T$ if and only if (\ref{16}) holds.

In view of the collection of all eigenvalues of $A$, we know that it either  comprises of a finite set or  it only has an accumulative point  zero.
In the second case we see, according to (\ref{15}) and the fact $\lambda>-\gamma$, that (\ref{14}) has a unique weak solution for all
$g\in L^2_T$, except for a sequence $\lambda_j\rightarrow+\infty.$

\end{proof}

Similar existence theorems can be found in \cite{evans} for Dirichlet boundary value problem in the standard Laplacian case.

\section{lower bound of periods}

The authors and  Zhang in \cite{gzd} obtained the existence of periodic solutions with any period $T>T_0$ to (\ref{3}).
However,  the optimal value  of $T_0$ was not determined in \cite{gzd}.
Denote
$$
T_0:=\inf\{T: (\ref{3})~ \mbox{has periodic solutions with period } ~T\},
$$
we will prove that  $T_0\leq2\pi\times \left(\frac{1}{-F''(0)}\right)^{\frac{1}{2s}}$ in this section by using Hopf bifurcation theory.

To this end we need to consider the following eigenvalue problem
\begin{equation}\label{17}
(-\partial_{xx})^s  \varphi(x)  =\lambda \varphi(x),~~~~\varphi(x+2\pi)=\varphi(x).
 \end{equation}
We know that its first eigenvalue is zero and corresponding eigenfunctions are any constant-value functions.
The eigenfunctions correspond to the other eigenvalues must change sign  at least once in one period.
So we may suppose that $\varphi(0)=0$ in (\ref{17}).
The following lemma shows the simplicity property of  all eigenvalues of (\ref{17}).

\begin{lem}\label{lem2}
All nonzero eigenvalues of (\ref{17}) are
$$
\lambda_{m+1}=m^{2s},~~m=1,2,\cdots.
$$
Moreover, they are all simple and the space of eigenfunctions of  $\lambda_{m+1}$ is $\mbox{span}\{\varphi_{m+1}(x)\}=\mathbb{R} \sin(mx)$.
\end{lem}
\begin{proof}
Suppose $J_\alpha(y)$ and $I_\alpha(y)$ are the second and first   modified Bessel functions respectively, namely they are two linearly independent
 solutions to the modified Bessel's equation
$$
y^2v''(y)+yv'(y)-(y^2+\alpha^2)v(y)=0.
$$
It is well-known that
\begin{equation}\label{18}
 I_\alpha(y)=\sum_{j=0}^\infty \frac{1}{j!\Gamma(j+\alpha+1)}\left(\frac{y}{2}\right)^{2j+\alpha}
\end{equation}
and
\begin{equation}\label{19}
J_\alpha(y)=\frac{\pi}{2}\frac{I_{-\alpha}(y)-I_\alpha(y)}{\sin(\alpha \pi)},
\end{equation}
when $\alpha$ is not an integer.
For $m\in N$, we set
\begin{equation}\label{20}
J_m(y):=\mu y^\gamma J_s(m y),
\end{equation}
where the parameters  $\gamma$ and $\mu$ will be determined later.
Elementary computation shows that
\begin{eqnarray}\nonumber
&&J''_m(y)+\frac{a}{y}J'_m(y)-m^2 J_m(y)\\
\nonumber =&&
\mu y^{\gamma-2}\{(my)^2J''_s(my)+(2\gamma+a)(my)J'_s(my)
\\
\nonumber &&
+[\gamma(\gamma-1)+a\gamma-(my)^2]J_s(my)\}.
\end{eqnarray}
We set $\gamma=\frac{1-a}{2}=s$, then $\gamma(\gamma-1)+a\gamma=-\gamma^2=-s^2$.
Hence
$$
J''_m(y)+\frac{a}{y}J''_m(y)-m^2 J_m(y)=0.
$$
Simple computation shows that
$J_m(0)=\mu \frac{\pi2^{s}m^{-s}}{2\Gamma(1-s)\sin(s \pi)}$.
Hence we can obtain $J_m(0)=1$, provided that the parameter $\mu$ be chosen as $\frac{2^{1-s}\Gamma(1-s)m^s\sin(s \pi)}{\pi }$.
From (\ref{18})-(\ref{20}), for any positive integer $m$, we can verify that the following limit exists and we denote it as
$$
\lim_{y\downarrow0}y^a\frac{\partial J_m(y)}{\partial y}=:-C_s\lambda_{m+1},
$$
where
$$
\lambda_{m+1}=m^{2s},
~~C_s=\frac{2s}{2^{2s}}\times\frac{\Gamma(1-s)}{\Gamma(1+s)}.
$$
  Here
we have used the fact $s\in(0,1)$.
Note that $C_s=\frac{1}{d_s}$.

 For a non-constant periodic solution $\varphi$ of (\ref{17}),
one has
$\int_0^{2\pi}\varphi(x)dx=0,$
since it is orthogonal with constant-value functions in $L^2$ sense.
By the completeness of orthogonal basis $\{\sin(mx),~\cos(mx)\}|_{m=0}^\infty$ in $L_{2\pi}^2$, we can write
$$
\varphi(x)=\sum_{m=1}^\infty [a_m\sin(mx)+b_m\cos(mx)].
$$
Now we define
$$
\Phi(x,y)=\sum_{m=1}^\infty J_m(y)[a_m\sin(mx)+b_m\cos(mx)],
$$
then $\Phi$ satisfies
$$
\begin{cases}
\mbox{div}(y^a\nabla    \Phi)=0 \indent &\mbox{in}~~
\mathbb{R}_+^{2},\\
 \Phi(x,0)=\varphi(x).\\
\end{cases}
$$
Therefore
$$
(-\partial_{xx})^s \varphi(x)=\lim_{y\downarrow0}-d_sy^a\frac{\partial \Phi}{\partial y}=\sum_{m=1}^\infty \lambda_{m+1}[a_m\sin(mx)+b_m\cos(mx)].
$$
This and (\ref{17}) give that
$$
\sum_{m=1}^\infty \lambda_{m+1}[a_m\sin(mx)+b_m\cos(mx)]=\sum_{m=1}^\infty \lambda[a_m\sin(mx)+b_m\cos(mx)].
$$
Hence  all non-zero eigenvalues of (\ref{17}) are exactly $\lambda_{m+1}(m=1,2,\cdots)$ and they are all simple.
In fact, the space of eigenfunctions of the $(m+1)$-eigenvalue $\lambda_{m+1}$ is $\mbox{span}\{\varphi_{m+1}(x)\}=\mathbb{R} \sin(mx)$,
recalling that $\varphi_{m+1}(0)=0$.

\end{proof}

Similar arguements also lead to the following  general proposition about eigenvalues and eigenfunctions for Schrodinger type operators involving fractional Laplacians.

\begin{thm}
Consider  a classical  eigenvalue problem for  a Hilbert space ${\mathcal H}$
\begin{equation}\label{general-eig}
-\Delta  \phi + V(x) \phi = \lambda \phi,
\end{equation}
where $V\geq0$. If the eigenfunctions $\phi_{m} $ with eigenvalue $\lambda_m,  m = 1, 2, \cdots,$   form a complete orthonormal  basis of ${\mathcal H}$,  then

\begin{equation}\label{fractional-eig}
[(-\Delta )+ V(x)]^s \phi = \lambda \phi
\end{equation}
has the same eigenfunctions $\phi_{m} $ with eigenvalues
$$
\lambda_{s, m}=(\lambda_m)^{s}, \quad m=1,2,\cdots.
$$
In particular,  if $\lambda_m$ is simple,  then $\lambda_{s, m}$ is  simple.
\end{thm}

Next we shall study the periodic solution of the nonlinear problem.

\begin{thm}\label{thm4}
Assume that $F$ satisfies conditions (\ref{4})-(\ref{5}) and
$F''(0)<0.$ Then  $T_0\leq2\pi\times \left(\frac{1}{-F''(0)}\right)^{\frac{1}{2s}}$.
\end{thm}

\begin{proof}
 For $\lambda>0$, let $x=\left(\frac{\lambda}{-F''(0)}\right)^{\frac{1}{2s}}\bar{x}$ and
$$
u(x)=u\left(\left(\frac{\lambda}{-F''(0)}\right)^{\frac{1}{2s}}\bar{x}\right)=:\bar{u}(\bar{x}),
$$
then (\ref{3}) is equivalent to
$$
(-\partial_{\bar{x}\bar{x}})^s  \bar{u}(\bar{x})+\frac{\lambda}{-F''(0)}F'(\bar{u}(\bar{x}))=0,\indent
        \bar{x}\in \mathbb{R}.
$$
For simplicity, we write this equation as
\begin{equation}\label{22}
    (-\partial_{xx})^s  u(x)+\frac{\lambda}{-F''(0)}F'(u(x))=0,\indent
        x\in \mathbb{R}.
\end{equation}

We want to  prove that (\ref{22}) admits periodic solutions  with period $2\pi$   for  $\lambda>\lambda_2=1$.
If this is done, then we obtain periodic solutions of (\ref{3}) with period $T>2\pi\times \left(\frac{1}{-F''(0)}\right)^{\frac{1}{2s}}.$
Then we obtain that the lower bound $T_0$ of periods  satisfies $T_0\leq2\pi\times \left(\frac{1}{-F''(0)}\right)^{\frac{1}{2s}}.$

Set functional
$$
G(\lambda, u):= (-\partial_{xx})^s  u(x)+\frac{\lambda}{-F''(0)}F'(u(x)).
$$
We define
$$
H^s_{T,0}:=\{u\in H^s_T, u(0)=0\},~~L^2_{T,0}:=\{u\in L^2_T, u(0)=0\}.
$$
We have $G(\lambda, u)\in C^2(\mathbb{R}\times H^s_{2\pi,0}, L^2_{2\pi,0})$.
Since $F$ satisfies condition (\ref{5}), we see that
$$
G(\lambda, 0)=0,~~~\forall~\lambda\in\mathbb{R}.
$$
Clearly
$$
G_u(\lambda, 0)= (-\partial_{xx})^s-\lambda.
$$
We next show that $\lambda_2$ is a bifurcation point for $G$. To this end we need to prove that $G_u(\lambda_2, 0)$ is a Fredholm map with one-dimensional kernel and index zero. Set
$$
V:=\mbox{ker}(G_u(\lambda_2, 0)),~R:=R(G_u(\lambda_2, 0)).
$$
From Lemma \ref{lem2} we know that $\dim(V)=1$ and $V=\mbox{span}\{\varphi_2(x)\}=\mathbb{R} \sin x$.
Choose $\varphi_2(x)=\frac{\sin x}{\sqrt{\pi}}$ so that $\int_{-\pi}^\pi  \varphi^2_2 dx=1$.

By Fredholm Alternative of compact operators, one deduces that $R=\{\mbox{ker}(G_u(\lambda_2, 0))\}^\perp$, which yields $\mbox{codim}(R)=1.$

Observe that
$
G_{u\lambda} (\lambda_2, 0)=-I,
$
so $G_{u\lambda} (\lambda_2, 0)\varphi_2=-\varphi_2\not\in R.$

From the classic bifurcation theory (see Theorem 1.7 in  \cite{crandall}),
we know that $\lambda_2$ is a bifurcation point for $G$.

By the  Rabinowitz global bifurcation
theorem(\cite{rabinowitz}), either the bifurcating branch is unbounded, or it meets another eigenvalue $\lambda_m(m\neq2)$ of the operator $G_u(\lambda,0)$.
We can rule out the latter case, so we conclude that the bifurcating branch is unbounded. Since $F$ satisfies conditions (\ref{4})-(\ref{5}), then any periodic solution $u$ of (\ref{22}) must have  $|u|<1$, hence we deduce that the first component $\lambda$ in the bifurcating branch $(\lambda,u)$ must increasing to positive infinity. This implies that there exists a periodic solution for
$\lambda>\lambda_2$

\end{proof}

\begin{rem}

1). The other eigenvalues $\lambda_{m+1}(m>1)$ are also bifurcation points of $G$.
Moreover, we can obtain the same result of Theorem \ref{thm4}, by dealing with these bifurcation points $\lambda_{m+1}$.
The difference is that we need to  prove that (\ref{22}) admits periodic solutions  with period $\frac{2\pi}{m}$   for  $\lambda>\lambda_{m+1}$.
From the proof of Theorem \ref{thm4}, we know that (\ref{3}) admits periodic solutions   with period $T>\frac{2\pi}{m}\times \left(\frac{\lambda_{m+1}}{-F''(0)}\right)^{\frac{1}{2s}}.$
Combining this and the result $\lambda_{m+1}=m^{2s}$, we also obtain $T_0\leq2\pi\times \left(\frac{1}{-F''(0)}\right)^{\frac{1}{2s}}$.

2). When $s=1$, one has $T_0=2 \pi /  \sqrt{-F''(0)}$ (see \cite{gzd}).
Furthermore, for the specific nonlinear function $F(u)=\frac{(1-u^2)^2}{4}$, we see that $F''(0)=-1$, which yields $T_0=2\pi$ (see \cite{Ambrosetti}, chapter 5).

3). Furthermore if $F$ satisfies
$F^{(3)}(0)=0$ and $F^{(4)}(0)>0,$
then any bifurcation points are supercritical locally.
\end{rem}

Next we shall only specify the behavior of the bifurcating branch near $(\lambda_2, 0)$, since the other bifurcation points are similar.

Since $\mbox{codim}(R)=1$,  there exists a linear functional $\varsigma\in (L^2_{2\pi,0})^\ast,~\varsigma\neq0$, such that
$$
R=\left\{u\in L^2_{2\pi,0}: \langle \varsigma, u\rangle=0\right\}.
$$
On the other hand, by $R=V^\perp$, we have
$$
R=\left\{u\in L^2_{2\pi,0}: \int_{-\pi}^\pi u \varphi_2dx=0\right\}.
$$
Hence we can define
$$
\langle \varsigma, u\rangle:=\int_{-\pi}^\pi u \varphi_2 dx.
$$

Note that $G_{uu}(\lambda, u)=\frac{\lambda}{-F''(0)}F^{(3)}(u)$, then by the condition $F^{(3)}(0)=0$, we have $G_{uu}(\lambda_2, 0)=0$, which yields to
$$
\langle \varsigma, G_{uu}(\lambda_2, 0)[\varphi_2,\varphi_2] \rangle=0.
$$
This eliminates the transcritical case, namely the first component $\lambda$ in the bifurcating branch $(\lambda,u)$ has only two possibilities:
$\lambda>\lambda_2$ (supercritical) or $\lambda<\lambda_2$ (subcritical).
We will show that it must be supercritical case.

Since $G_{u\lambda} (\lambda_2, 0)=-I,$
we have
$$
\langle \varsigma, G_{u\lambda} (\lambda_2, 0)\varphi_2 \rangle=\langle \varsigma, -\varphi_2 \rangle=-\int_{-\pi}^\pi \varphi^2_2 dx=-1.
$$
We also have
$$
\langle \varsigma, G_{uuu}(\lambda_2, 0)[\varphi_2]^3 \rangle=\frac{\lambda_2 F^{(4)}(0)}{-F''(0)}\int_{-\pi}^\pi  \varphi^4_2 dx>0.
$$
Therefore
$$
\frac{\langle \varsigma, G_{uuu}(\lambda_2, 0)[\varphi_2]^3 \rangle}{-6\langle \varsigma, G_{u\lambda} (\lambda_2, 0)\varphi_2 \rangle}
=
\frac{\langle \varsigma, G_{uuu}(\lambda_2, 0)[\varphi_2]^3 \rangle}{6}>0.
$$
By the standard bifurcation theory (see Chapter 5, \cite{Ambrosetti}), we obtain the supercriticality  result.

\section{Hamiltonian estimates}

We will first establish Hamiltonian identity for periodic solutions of (\ref{3}).
Similar Hamiltonian identity can be found in \cite{cabre2} and \cite{gz}.

\begin{thm}\label{thm5}(Hamiltonian identity)
Assume $U$ is the $s$-harmonic extension of a periodic solution $u$ of (\ref{3}). Then for all  $x\in\mathbb{R}$ we have
$$
   \frac{1}{2}\int_0^\infty [U_x^2(x,y)-U_y^2(x,y)]y^{a} dy-F(U(x,0)) \equiv  C_T.
$$
\end{thm}
\begin{proof}
By Lemma 5.1 in \cite{cabre2}, we have $\int_0^{+\infty}y^a|\nabla U(x,y)|^2dy<\infty$.
Hence $\lim_{y\rightarrow +\infty}y^aU_y(x,y)U_x(x,y)=0.$

We introduce the function
$$
   w(x):=\frac{1}{2}\int_0^\infty [U_x^2(x,y)-U_y^2(x,y)]y^{a} dy.
$$
Regularity result (see Lemma 5.1 in \cite{cabre2})  allows us to differentiate within the integral in  the above equality to get
$$
w'(x)=\int_0^\infty y^{a}[U_xU_{xx}-U_yU_{xy}](x,y) dy.
$$
Note that
$$
(y^aU_y)_y+y^aU_{xx}=0.
$$
Using integration by parts, we have
$$
w'(x)=-[y^aU_y(x,y)U_x(x,y)]|_{y=0}^{+\infty}=\lim_{y\rightarrow0^+}y^aU_y(x,y)U_x(x,y).
$$
Since $U$ is the $s$-harmonic extension of  solution $u$ of (\ref{3}), we have
$$
\lim_{y\rightarrow0^+}y^aU_y(x,y)U_x(x,y)=(-\partial_{xx})^su(x)u'(x)=\frac{d}{dx}F(U(x,0)).
$$
Hence
$$
w'(x)=\frac{d}{dx}F(U(x,0)),
$$
which gives the  result of this lemma.

\end{proof}

To set up the following Modica-type  inequalities, we assume that $F$ satisfies (\ref{4})-(\ref{5}) and is even.

The existence of an odd and periodic solution $u(x)$ of (\ref{3}) has been proved in \cite{gzd} by using  variational method.
Indeed, an even and and periodic solution can also been shown to exist by using the same  method.
Different from \cite{gzd}, now we consider  the energy functional
$$
J(U,\Omega_T):=\frac{1}{2}\int_{\Omega_T}y^a|\nabla U(x,y)|^2dxdy
+\int_{-\frac{T}{2}}^{\frac{T}{2}}F(U(x,0))dx
$$ in the admissible set
$$
\Lambda_T=\{U\in H^1(\Omega_T,y^a): U(x+T,y)=U(x,y), U(-x,y)=U(x,y),
$$
$$
 U(0,y)\geq0\geq U(\frac{T}{2},y),
~\forall ~ y\geq0\},
$$
where $\Omega_T:=[-\frac{T}{2},\frac{T}{2}]\times [0,+\infty)$ and
\begin{equation}\label{23}
H^1(\Omega_T,y^a):=\{U:  y^{a}(U^2+|\nabla U|^2)\in L^1(\Omega_T)\}.
\end{equation}
Note that $J(U,\Omega_T)\geq0$. On the other hand, we have that $0\in \Lambda_T$ and $J(0,\Omega_T)=F(0)T<+\infty$.
Hence there exists a minimizing sequence $\{U_k\}\subseteq \Lambda_T$ of $J$, namely
$$
\lim_{k\rightarrow\infty}J(U_k,\Omega_T)=m_T:=\inf_{U\in\Lambda_T}J(U,\Omega_T).
$$
From the condition of $F$ and the definition of $\Lambda_T$, we may assume that $\partial_x U_k\leq0$ in $[0,\frac{T}{2}]\times [0,+\infty)$.
Similarly as in \cite{gzd}, we can find a minimizer $U_T$ of the energy $J$ in $\Lambda_T$, and prove that $U_T\not\equiv0.$

Then we extend $U_T$  periodically  (with respect to $x$) from $\Omega_T$ to the whole half space $\overline{\mathbb{R}_+^{2}}$, and we still denote it as $U_T$.

Set
$$
u(x):=U_T(x,0).
$$
Then $u$ is an even periodic solution of (\ref{3}).
A Hopf principle in \cite{cabre2} shows that $U_T(x,0)=u(x)\in(-1,1)$ and $u'(x)<0$ in $(0,\frac{T}{2})$.

\begin{thm}\label{thm6}(Modica-type  inequalities)
Assume $U(x,y)$ is the $s$-harmonic extension of an even periodic solution $u(x)$ of (\ref{3}). Then for every $y\geq0$ and all  $x\in\mathbb{R}$ we have
\begin{equation}\label{24}
   \frac{1}{2}\int_0^y [U_x^2(x,\tau)-U_y^2(x,\tau)]\tau^{a} d\tau-F(U(x,0))-C_T\leq \hat{C},
\end{equation}
where   $\hat{C}:=\sup_{x\in\mathbb{R}}\{-F(u(x))-C_T\}>0$ and $C_T$ is the constant given in Theorem \ref{thm5}.
\end{thm}
\begin{proof}
We introduce the function
$$
v(x,y):=\frac{1}{2}\int_0^y [U_x^2(x,\tau)-U_y^2(x,\tau)]\tau^{a} d\tau
$$
and define
$$
\hat{v}(x,y):=\frac{1}{2}\int_0^y [U_x^2(x,\tau)-U_y^2(x,\tau)]\tau^{a} d\tau-F(U(x,0))-C_T.
$$
By the periodicity and even symmetry  of $U(x,y)$ (with respect to $x$), it suffices to prove (\ref{24}) for every $y\geq0$ and all  $x\in[0,\frac{T}{2}]$.
Note that
\begin{equation}\label{25}
\lim_{y\rightarrow+\infty}\hat{v}(x,y)=0.
\end{equation}
Recall that $U_x(0,y)=0=U_x(\frac{T}{2},y)$ for any $y\geq0$, we have
\begin{equation}\label{26}
\hat{v}(0,y)<\hat{v}(0,0),~~~\hat{v}(\frac{T}{2},y)<\hat{v}(\frac{T}{2},0).
\end{equation}
Hence $\hat{v}$ is not identically constant.

Elementary calculation shows $\hat{v}_x=-y^aU_xU_y$ and
 \begin{equation}\label{27}
\mbox{div}(y^{-a}\nabla    \hat{v})=ay^{-1}U_y^2.
\end{equation}
Owing to $U_x<0$ in $(0,\frac{T}{2})\times(0,+\infty)$,
equation (\ref{27}) can be written as
$$
\mbox{div}(y^{-a}\nabla    \hat{v})+ay^{-1-a}\frac{U_y}{U_x}\hat{v}_x=0.
$$
Note that the operator in the left hand side is uniformly elliptic with continuous coefficients in compact sets of $(0,\frac{T}{2})\times(0,+\infty)$.
Since $\hat{v}$ is not identically constant,
 $\hat{v}$ cannot achieve its maximum  in any interior point of $(0,\frac{T}{2})\times(0,+\infty)$.
This fact and (\ref{25})-(\ref{26}) show that $\hat{v}$ achieves its maximum  in  $[0,\frac{T}{2}]\times[0,+\infty)$ at $[0,\frac{T}{2}]\times \{0\}$,
and we denote the maximum as $\hat{C}$.
We have
\begin{eqnarray}
\nonumber
\hat{C}&=&\sup\{ -F(U(x,0))-C_T\}\geq -F(U(\frac{T}{2},0))-C_T
\\ \nonumber &
=&\frac{1}{2}\int_0^{+\infty} U_y^2(\frac{T}{2},\tau) d\tau>0.
\end{eqnarray}
The proof is complete.

\end{proof}

\section{Asymptotic behavior}

The following results are established in \cite{gzd}.
\begin{pro}\label{pro2}
(\cite{gzd}) Let $s\in(0,1)$. Assume  $F$ satisfies the assumptions (\ref{4})-(\ref{5}) and is even. Then there exists $T_0>0$ such that
for any $T>T_0$,  Eq.(\ref{3}) admits an odd periodic solution $u_T$ with period $T$, and $u_T(x)\in(0,1)$ for $x\in(0,\frac{T}{2})$. Moreover,
for any positive number $\sigma<\frac{1}{2}$, there exists $T_\sigma\geq T_0$ such that for any  $T>T_\sigma$, we have
 \begin{equation}\label{28}
J(U_T,\Omega_T)<\sigma F(0)T,
\end{equation}
where $U_T$ is the $s$-harmonic extension of $u_T$.
\end{pro}

\begin{rem}
 From Theorem 3.1 we know that $T_0\leq2\pi\times \left(\frac{1}{-F''(0)}\right)^{\frac{1}{2s}}$ under the further condition $F''(0)<0$.
 \end{rem}

We recall the main idea for the proof of Proposition \ref{pro2} in \cite{gzd}.
The $s$-harmonic extension  $U_T$ of solution $u_T$ to (\ref{3}) corresponds to  energy functional
\begin{equation}\label{29}
J(U,\Omega_T)=\frac{1}{2}\int_{\Omega_T}y^a|\nabla U(x,y)|^2dxdy
+\int_0^{\frac{T}{2}}F(U(x,0))dx.
\end{equation}
We denote the admissible set of the energy $J$ as
$$
\Lambda_T:=\{U:  U\geq0, U(0,y)=0=U(\frac{T}{2},y),~ \forall y\geq0, U\in H^1(\Omega_T,y^a)\},
$$
where the notations $\Omega_T, ~H^1(\Omega_T,y^a)$ are defined in (\ref{23}).
  The existence of nontrivial minimizer $U_T$ of $J$  in $\Lambda_T$  is  obtained
  in \cite{gzd}, togeter with the estimate (\ref{28}).

Next we  shall improve  estimate (\ref{28})  in Proposition \ref{pro2}.

\begin{thm}\label{thm8}
Under the conditions in Proposition \ref{pro2}, we have
\begin{equation}\label{30}
J(U_T,\Omega_T)\leq\begin{cases}
CT^{1-2s},&s\in(0, \frac{1}{2}),\\
C\ln T,&s=\frac{1}{2},\\
C,&s\in (\frac{1}{2},1).\\
\end{cases}
\end{equation}
\end{thm}
\begin{proof}
Recall that in Section 2  we have
$$
\int_{-\frac{T}{2}}^{\frac{T}{2}}\int_{\mathbb{R}^+}y^a|\nabla U(x,y)|^2dxdy=\frac{C(s)}{2}\int_{-\frac{T}{2}}^{\frac{T}{2}}\int_{\mathbb{R}}\frac{|u(x)-u(\bar{x})|^2}{|x-\bar{x}|^{1+2s}}d\bar{x} dx.
$$
We construct  the following continuous function
$$
h(x):=
\begin{cases}
\frac{x}{d},&x\in[0, d],\\
1,&x\in[d, \frac{T}{2}-d],\\
-\frac{1}{d}(x-\frac{T}{2}),&x\in [\frac{T}{2}-d, \frac{T}{2}],\\
\end{cases}
$$
for some constant $d$.
We extend $h$ oddly   from $[0,\frac{T}{2}]$ to $[-\frac{T}{2},\frac{T}{2}]$. Further we extend it periodically  with period $T$.
We still denote this periodic solution as $h$.
It is easy to verify that there exists a function $H(x,y)\in \Lambda_T$ such that $h(x)=H(x,0)$.
Therefore, to prove (\ref{30}), it is enough to show that
\begin{eqnarray}\label{31}
&&\int_{-\frac{T}{2}}^{\frac{T}{2}}\int_{\mathbb{R}}\frac{|h(x)-h(\bar{x})|^2}{|x-\bar{x}|^{1+2s}}d\bar{x} dx+\int_0^{\frac{T}{2}}F(h(x))dx
\\
\nonumber &&
\leq\begin{cases}
CT^{1-2s},&s\in(0, \frac{1}{2}),\\
C\ln T,&s=\frac{1}{2},\\
C,&s\in (\frac{1}{2},1).\\
\end{cases}
\end{eqnarray}
To obtain (\ref{31}), we need to prove that
\begin{eqnarray}\label{32}
&&\int_{-\frac{T}{2}}^{\frac{T}{2}}\int_{|x-\bar{x}|\geq \frac{T}{2}}\frac{|h(x)-h(\bar{x})|^2}{|x-\bar{x}|^{1+2s}}d\bar{x} dx
\\
\nonumber &&
+
\int_{-\frac{T}{2}+d}^{-d}\int_d^{ \frac{T}{2}-d}\frac{|h(x)-h(\bar{x})|^2}{|x-\bar{x}|^{1+2s}}d\bar{x} dx
\\
\nonumber &&
+
\int_{-\frac{T}{2}+d}^{-d}\int_{-d}^d\frac{|h(x)-h(\bar{x})|^2}{|x-\bar{x}|^{1+2s}}d\bar{x} dx
\\
\nonumber &&
+
\int_{-d}^{d}\int_{-d}^d\frac{|h(x)-h(\bar{x})|^2}{|x-\bar{x}|^{1+2s}}d\bar{x} dx
\leq\begin{cases}
CT^{1-2s},&s\in(0, \frac{1}{2}),\\
C\ln T,&s=\frac{1}{2},\\
C,&s\in (\frac{1}{2},1).\\
\end{cases}
\end{eqnarray}
For the first integral, we have
\begin{equation}\label{33}
\int_{-\frac{T}{2}}^{\frac{T}{2}}\left(\int_{|x-\bar{x}|\geq \frac{T}{2}}\frac{|h(x)-h(\bar{x})|^2}{|x-\bar{x}|^{1+2s}}d\bar{x} \right)dx
\leq
CT^{1-2s}.
\end{equation}
For the second integral, we have
\begin{eqnarray}\label{34}
&&\int_{-\frac{T}{2}+d}^{-d}\int_d^{ \frac{T}{2}-d}\frac{|h(x)-h(\bar{x})|^2}{|x-\bar{x}|^{1+2s}}d\bar{x} dx
\\
\nonumber &&
\leq\frac{4}{2s}\int_{-\frac{T}{2}-d}^{-d}|d-x|^{-2s} dx\\
\nonumber &&
\leq\begin{cases}
CT^{1-2s},&s\in(0, \frac{1}{2}),\\
C\ln T,&s=\frac{1}{2},\\
C,&s\in (\frac{1}{2},1).\\
\end{cases}
\end{eqnarray}
For the third integral, we have
\begin{eqnarray}\label{35}
&&\int_{-\frac{T}{2}+d}^{-d}\left(\int_{-d}^d\frac{|h(x)-h(\bar{x})|^2}{|x-\bar{x}|^{1+2s}}d\bar{x}\right) dx
\\
\nonumber &&
=d^{-2}\int_{-\frac{T}{2}+d}^{-d}\left(\int_{-d}^d\frac{|d+\bar{x}|^2}{|x-\bar{x}|^{1+2s}}d\bar{x}\right) dx
\\
\nonumber &&
=d^{-2}\int_{-d}^d\left(\int_{-\frac{T}{2}+d}^{-d}\frac{|d+\bar{x}|^2}{|x-\bar{x}|^{1+2s}}dx\right) d\bar{x}
\\
\nonumber &&
\leq Cd^{-2}\int_{-d}^{d}|d+\bar{x}|^{2} |d+\bar{x}|^{-2s}d\bar{x}\\
\nonumber &&
\leq
C.
\end{eqnarray}
For the last integral, we have
\begin{eqnarray}\label{36}
&&\int_{-d}^{d}\left(\int_{-d}^d\frac{|h(x)-h(\bar{x})|^2}{|x-\bar{x}|^{1+2s}}d\bar{x}\right) dx
\\
\nonumber &&
\leq d^{-2}\int_{-d}^{d}\int_{-d}^d|x-\bar{x}|^{1-2s}dxd\bar{x}\\
\nonumber &&
\leq Cd^{1-2s}
\leq C.
\end{eqnarray}
Combining inequalities (\ref{33})-(\ref{36}), we obtain (\ref{32}).  The proof is
complete.
\end{proof}

Note that similar energy estimates are obtained in \cite{Valdinoci} for minimizers of the functional   in a finite interval $[a, b]$  with a homogeneous   condition outside the interval instead of a periodic condition.

{\bf Acknowledgment.} The first author is  supported by   the Natural Science Foundation of Hunan Province, China
(Grant No. 2016JJ2018). The second author is supported by  NSF DMS-1601885.


\begin{thebibliography}{CL}

\bibitem{Ambrosetti}A. Ambrosetti and G. Prodi,  {\it A primer of nonlinear analysis},
 Cambridge Studies in Advanced Mathematics, vol. 34, Cambridge University Press, Cambridge, 1993.

\bibitem{Ambrosio1}V. Ambrosio,  {\it Periodic solutions for a pesudo-relativistic Schr\"{o}dinger equation},
Nonlinear Anal., 120(2015), 262-284.

\bibitem{Ambrosio2}V. Ambrosio,  {\it Periodic solutions for the non-local operator $(-\Delta+m^2)^{s}-m^{2s}$ with $m\geq0$},
 Topol. Methods Nonlinear Anal., Volume 120, June 2015, Pages 262-284

\bibitem{Ambrosio3} V. Ambrosio and G.M. Bisci,  {\it Periodic solutions for nonlocal fractional equations}, Communications On Pure And Applied Analysis, 
Volume 16, Number 1, January 2017

\bibitem{cabre2}
X. Cabr\'{e} and Y. Sire,  {\it Nonlinear  equations for fractional Laplacians I: regularity, maximum principles, and
Hamiltonian estimates},
 \,\,Ann. Inst. H. Poincare Analyse Non linneaire, 31(1)(2014), 23-53.


\bibitem{caff}   L. Caffarelli and L. Silvestre,  {\it An extension problem related to the fractional Laplacian},
 Comm. Partial Differential Equations, 32(2007), 1245-1260.

\bibitem{crandall} M. Crandall and P. Rabinowitz,  {\it Bifurcation from simple eigenvalues},
 J. Funct. Anal., 8(1971), 321-340.

\bibitem{ddgw}   A. de la Torre, M. del Pino, Mar del mar Gonzalez and J. Wei,  {\it Delaunay-type singular solutions for the fractional Yamabe problem},
   Math. Ann. 369 (2017), no. 1-2, 597-626. 

\bibitem{evans}   L. C. Evans,  {\it Partial differential equations}, American Mathematical Society Procidence, Rhode Island.


\bibitem{gzd}   C. Gui, J. Zhang and Z. Du,  {\it Periodic solutions   of  a  semilinear elliptic equation with fractional Laplacian},
  J. Fix. Point Theory A, 19(1)(2017), 363-373.

\bibitem{gz}   C. Gui and M. Zhao,  {\it Traveling wave solutions   of    Allen-Cahn equation with a fractional Laplacian},
  Ann. I. H. Poincar\'{e}, 32(4)(2015), 785-812.
  
  
\bibitem{Valdinoci}   G.Palatucci, O. Savin and E. Valdinoci,  {\it Local and global minimizers for a variational energy involving a fractional norm},
   Annali di Matematica Pura ed Applicata, 192 (2013), no. 4, 673-718.

\bibitem{rabinowitz}  P. Rabinowitz,  {\it Some global results for nonlinear eigenvalue problems},
 J. Funct. Anal., 7(1971), 487-513.
 
\bibitem{roncal}  L. Roncal and P. R. Stinga,  {\it Fractional Laplacian on the torus},
 Commun. Contemp. Math., 18, 1550033(2016), [26 pages]
https://doi.org/10.1142/S0219199715500339.


\end{thebibliography}
\end{document}